\documentclass[a4paper,12pt]{article}
\usepackage{lineno}

\usepackage{amssymb}
\begin{document}

\newtheorem{theorem}{Theorem}[section]
\newtheorem{corollary}[theorem]{Corollary}
\newtheorem{definition}[theorem]{Definition}
\newtheorem{conjecture}[theorem]{Conjecture}
\newtheorem{question}[theorem]{Question}
\newtheorem{lemma}[theorem]{Lemma}
\newtheorem{proposition}[theorem]{Proposition}
\newtheorem{example}[theorem]{Example}
\newtheorem{problem}[theorem]{Problem}
\newenvironment{proof}{\noindent {\bf
Proof.}}{\rule{3mm}{3mm}\par\medskip}
\newcommand{\remark}{\medskip\par\noindent {\bf Remark.~~}}
\newcommand{\pp}{{\it p.}}
\newcommand{\de}{\em}

\newcommand{\JEC}{{\it Europ. J. Combinatorics},  }
\newcommand{\JCTB}{{\it J. Combin. Theory Ser. B.}, }
\newcommand{\JCT}{{\it J. Combin. Theory}, }
\newcommand{\JGT}{{\it J. Graph Theory}, }
\newcommand{\ComHung}{{\it Combinatorica}, }
\newcommand{\DM}{{\it Discrete Math.}, }
\newcommand{\ARS}{{\it Ars Combin.}, }
\newcommand{\SIAMDM}{{\it SIAM J. Discrete Math.}, }
\newcommand{\SIAMADM}{{\it SIAM J. Algebraic Discrete Methods}, }
\newcommand{\SIAMC}{{\it SIAM J. Comput.}, }
\newcommand{\ConAMS}{{\it Contemp. Math. AMS}, }
\newcommand{\TransAMS}{{\it Trans. Amer. Math. Soc.}, }
\newcommand{\AnDM}{{\it Ann. Discrete Math.}, }
\newcommand{\NBS}{{\it J. Res. Nat. Bur. Standards} {\rm B}, }
\newcommand{\ConNum}{{\it Congr. Numer.}, }
\newcommand{\CJM}{{\it Canad. J. Math.}, }
\newcommand{\JLMS}{{\it J. London Math. Soc.}, }
\newcommand{\PLMS}{{\it Proc. London Math. Soc.}, }
\newcommand{\PAMS}{{\it Proc. Amer. Math. Soc.}, }
\newcommand{\JCMCC}{{\it J. Combin. Math. Combin. Comput.}, }
\newcommand{\GC}{{\it Graphs Combin.}, }

\title{Extremal Graph Theory  
 for Degree Sequences\thanks{
Supported by National Natural Science Foundation of China
(No.11271256) and  Innovation Program of Shanghai Municipal Education Commission (No:14ZZ016)}}
\author{  Xiao-Dong Zhang \\
{\small Department of Mathematics and MOE-LSC}\\
{\small Shanghai Jiao Tong University} \\
{\small  800 Dongchuan road, Shanghai, 200240, P.R. China}\\
{\small Email:  xiaodong@sjtu.edu.cn}
 }
\date{}
\maketitle
 \begin{abstract}
 This paper surveys some recent results and progress on the extremal problems in a given set consisting of all simple connected graphs with the same graphic degree sequence. In particular, we study and characterize  the extremal graphs having the maximum (or minimum) values of graph invariants such as (Laplacian, p-Laplacian, signless Laplacian) spectral radius, the first Dirichlet eigenvalue, the Wiener index, the Harary index, the number of subtrees and  the chromatic number etc, in given  sets with the same tree, unicyclic, graphic degree sequences. Moreover, some conjectures are included.
 \end{abstract}

{{\bf Key words:} Graphic degree sequence, majorization, tree, unicyclic graphs, spectral radius, Wiener index.
 }

      {{\bf AMS Classifications:} 05C50, 05C05, 05C07, 05C12, 05C35}
\vskip 0.5cm

\section{Introduction}

 A nonincreasing sequence of nonnegative integers
 $\pi=(d_0,d_1,\cdots, d_{n-1})$  with $d_0\ge d_1\ge \cdots \ge d_{n-1}$ is called {\it graphic} if there
 exists a simple graph having $\pi$ as its  vertex degree sequence and such a graph is called a {\it realization} of $\pi$.
Erd\"{o}s and Gallai \cite{erdos1960} characterized all graphic degree sequences.
\begin{theorem}\cite{erdos1960} \label{erdos-th}
A nonincreasing sequence of nonnegative integers
$\pi=(d_0,d_1,\cdots,$ $ d_{n-1})$
   is graphic  degree sequence if and only if  $\sum_{i=0}^{n-1}d_i$ is even and
   $$\sum_{i=0}^rd_i\le r(r+1)+\sum_{i=r+1}^{n-1}\min\{r+1, d_i\}, $$
for $r=0, 1,\cdots, n-2.$
\end{theorem}
Moreover, Senior \cite{senior1951} and Hakimi\cite{hakimi1962}  characterized  all graphic degree sequences with at least a realization being connected.
\begin{theorem}(\cite{senior1951},\cite{hakimi1962})\label{sen-hak}
A nonincreasing sequence of nonnegative integers
$\pi=(d_0, d_1,$ $\cdots, d_{n-1})$
   is graphic  sequence  with  at least a realization being connected if and only if  $\sum_{i=0}^{n-1}d_i$ is even and
   $$\sum_{i=0}^rd_i\le r(r+1)+\sum_{i=r+1}^{n-1}\min\{r, d_i\}, $$
for $r=0, 1,\cdots, n-2,$
$\sum_{i=0}^{n-1}d_i\ge \frac{n(n-1)}{2} $ and $d_{n-1}\ge 1$.
\end{theorem}
 Let $\pi$ be graphic degree sequence with at least a realization being connected. Denote by $\mathcal{G}(\pi)$ the set of all connected graphs which are realizations of $\pi$,  i.e.,
\begin{equation}
\mathcal{G}(\pi)=\{G \ |  G \ {\rm is\  connected \ and\ has \  \pi \ as
 \  its \  degree\  sequence}\}.
\end{equation}
Sierksma and Hoogeveen \cite{sierksma1991} surveyed seven criteria for a nonnegative integer sequence being graphic. For graphic sequences, Ferrara \cite{ferrara2014} surveyed   recent research progress and new results on graphic sequences and presented a
number of approachable open problems.

 Extremal graph theory is an extremal important branch of graph theory. One is
interested in relations among the various graph invariants,
such as order, size, connectivity, chromatic number, diameter
and eigenvalues, and also in the values of these invariants
which ensure that the graph has certain properties. Bollob\'{a}s published an excellent book \cite{bollobas1978} and recent survey in \cite{bollobas1995}.  Erd\"{o}s, Jackson and Lehel in \cite{erdos1991} proposed  the problem on the possible clique number attained by graphs with the same degree sequence, i.e., determine the smallest integer number $M$
 so that each graphic sequence $\pi=(d_0, \cdots, d_{n-1})$ with $d_0\ge d_1\ge\cdots\ge d_{n-1}\ge 1$ and $\sum_{i=0}^{n-1}d_i\ge M$ has a {\it realization} $G$ having given clique number. Li and Yin \cite{liyin2004} surveyed extremal graph theory and degree sequences, in particular for the above problem and its variants.

  On the other hand, many graphs coming from the real world have power law degree distribution: the number of vertices with degree $k$ is proportional for $k^{-\beta}$ for some exponent $\beta\ge 1$. Chung and Lu  (for example, see \cite{chunglu2002}, \cite{chunglu2006}) proposed the famous Chung-Lu's model: random graphs with given expected degree sequences and studied  their properties such as diameters, eigenvalues, the average distance, etc.
   In this survey, we just focus  on some extremal properties of graphs with given degree sequences. Graph invariants such as  the spectral radius, the Dirichlet eigenvalues, the Wiener index, the Harary index,   the number of subtrees, etc.,  will be considered. Graphic degree sequences such as tree degree sequences, unicyclic degree sequences etc, will be involved. Moreover, the majorization theorem on two different degree sequences are obtained.
Therefore, we may
 propose the following general problem.
\begin{problem} For a given  graphic degree sequence $\pi$, let
\begin{eqnarray*}
 {\mathcal G}_{ \pi}=\{ G \ | \   G \ {\rm is\  connected \ with\  \pi \ as
 \  its \  degree\  sequence}\}.
 \end{eqnarray*}
  For some graph invariants, such as  spectral radius, the Wiener index, etc,  find
   the maximum (minimum) value of these graph invariants in  ${\mathcal G}_{ \pi}$ and characterize all
extremal graphs which attain these values.
\end{problem}

\section{Preliminary}
In this section, we introduce some notions and properties. For a graph $G=(V, E)$
with a root $v_0$,  denote by $dist(v,v_0)$  the distance between
vertices $v$ and $v_0$. Moreover, the distance $dist(v,v_0)$ is
called the {\it height} $h(v)=dist(v, v_0)$ of a vertex $v$.

\begin{definition}\label{BFS}(\cite{Biyikoglu2008}, \cite{zhang2009})
Let $G=(V,E)$ be a graph with root $v_0$. A well-ordering $\prec$ of
the vertices is called a breadth-first search ordering (BFS-ordering
for short) if the following holds for all vertices $u, v\in V$:

(1) $u\prec v$ implies $h(u)\le h(v);$

 (2) $u\prec v$ implies $d(u)\ge d(v)$;

 (3) let  $ uv\in E(G)$,  $xy\in E(G)$, $uy\notin E(G), xv \notin
 E(G)$ with $h(u)=h(x)=h(v)-1=h(y)-1$. If  $u\prec x$,
 then $v\prec y$.\\
 We call a graph that has a {\it BFS-ordering} of its vertices a
 {\it BFS-graph}.
\end{definition}

For example, for a given degree sequence $\pi=(4, 4,
 3,3,3,3,2,1,1,1, 1, 1, 1,  1, 1, 1,1)$, $T_{\pi}^*$  is the tree of order 17
 (see Fig.1). There is a vertex $v_{01}$ in layer 0; four
 vertices $v_{11}, v_{12}, v_{13}, v_{14}$ in layer 1; nine
 vertices $v_{21},v_{22}, \cdots, v_{29}$ in layer 2; three
 vertices $v_{31}, v_{32}, v_{33}$ in layer 3.

 \setlength{\unitlength}{0.1in}
\begin{picture}(60,30)\label{fig11}
\put(28,27){\circle{0.5}}
 \put(28.7,27){$v_{01}$}
 \put(27.8,26.8){\line(-3,-1){20}}
\put(27.8,26.8){\line(-1,-1){6.6}}
 \put(28.2,26.8){\line(1,-1){6.6}}
\put(28.2,26.8){\line(3,-1){20}}

\put(7.6,20){\circle{0.5}} \put(20.9,20){\circle{0.5}}
\put(35,20){\circle{0.5}} \put(48.4,20){\circle{0.5}}

\put(9,19.5){$v_{11}$}\put(22,19.5){$v_{12}$}
\put(36,19.5){$v_{13}$}\put(49.4,19.5){$v_{14}$}

\put(7.4,19.8){\line(-1,-1){5.1}} \put(7.6,19.78){\line(0,-1){5}}
 \put(7.8,19.8){\line(1,-1){5.1}}
\put(2.2,14.5){\circle{0.5}} \put(7.6,14.5){\circle{0.5}}
\put(13,14.5){\circle{0.5}}

\put(20.7,19.8){\line(-1,-1){5.1}} \put(21.1,19.8){\line(1,-1){5.1}}
\put(15.4,14.5){\circle{0.5}} \put(26.4,14.5){\circle{0.5}}

\put(3.2,13.3){$v_{21}$} \put(8.5,13.3){$v_{22}$}
\put(12.6,13.3){$v_{23}$} \put(15.25,13.3){$v_{24}$}
\put(26.5,13.3){$v_{25}$} \put(29.5,13.3){$v_{26}$}
\put(40,13.3){$v_{27}$} \put(42.5,13.3){$v_{28}$}
\put(53,13.34){$v_{29}$}

 \put(34.8,19.8){\line(-1,-1){5.1}}
\put(35.2,19.8){\line(1,-1){5.1}}
 \put(29.5,14.5){\circle{0.5}}
\put(40.4,14.5){\circle{0.5}}

\put(48.2,19.8){\line(-1,-1){5.1}}
\put(48.6,19.8){\line(1,-1){5.1}}
 \put(43,14.5){\circle{0.5}}
\put(53.8,14.5){\circle{0.5}}

\put(2,14.3){\line(-1,-2){2}}
 \put(2.4,14.3){\line(1,-2){2}}
 \put(-0.15,10){\circle{0.5}}
\put(4.55,10){\circle{0.5}}

\put(7.6,14.25){\line(0,-1){3.8}}
 \put(7.6,10){\circle{0.5}}
 \put(-0.4,8.7){$v_{31}$} \put(4.1,8.7){$v_{32}$} \put(7, 8.7){$v_{33}$}

 \put(25,5){\bf Figure~1}
          \end{picture}
   \par
   It is easy to see that every tree has an ordering  such that it  satisfies the conditions (1) and (3) by using breadth-first search, but  not every tree has a
BFS-ordering. For example, the following tree $T$ of order 17 does not have  a
BFS-ordering with degree sequence $\pi=(4, 4,
 3,3,3,3,2,1,1,1,1,1, 1,$ $ 1, 1, 1,1)$ (see Fig.2).

 \setlength{\unitlength}{0.1in}
\begin{picture}(60,30)\label{fig2}

\put(28,27){\circle{0.5}}
 \put(28.7,27){$v_{01}$}
 \put(27.8,26.8){\line(-3,-1){20}}
\put(27.8,26.8){\line(-1,-1){6.6}}
 \put(28.2,26.8){\line(1,-1){6.6}}
\put(28.2,26.8){\line(3,-1){20}}

\put(7.6,20){\circle{0.5}} \put(20.9,20){\circle{0.5}}
\put(35,20){\circle{0.5}} \put(48.4,20){\circle{0.5}}

\put(9,19.5){$v_{11}$}\put(22,19.5){$v_{12}$}\put(36,19.5){$v_{13}$}\put(49.4,19.5){$v_{14}$}

\put(7.4,19.8){\line(-1,-1){5.1}} \put(7.6,19.78){\line(0,-1){5}}
 \put(7.8,19.8){\line(1,-1){5.1}}
\put(2.2,14.5){\circle{0.5}} \put(7.6,14.5){\circle{0.5}}
\put(13,14.5){\circle{0.5}}

\put(20.7,19.8){\line(-1,-1){5.1}} \put(21.1,19.8){\line(1,-1){5.1}}
\put(15.4,14.5){\circle{0.5}} \put(26.4,14.5){\circle{0.5}}

\put(3.2,13.3){$v_{22}$} \put(8.5,13.3){$v_{23}$}
\put(12.6,13.3){$v_{24}$} \put(16.5,13.3){$v_{21}$}
\put(26.5,13.3){$v_{25}$} \put(29.5,13.3){$v_{26}$}
\put(40,13.3){$v_{27}$} \put(42.5,13.3){$v_{28}$}
\put(53,13.34){$v_{29}$}

 \put(34.8,19.8){\line(-1,-1){5.1}}
\put(35.2,19.8){\line(1,-1){5.1}} \put(29.5,14.5){\circle{0.5}}
\put(40.4,14.5){\circle{0.5}}

\put(48.2,19.8){\line(-1,-1){5.1}} \put(48.6,19.8){\line(1,-1){5.1}}
 \put(43,14.5){\circle{0.5}}
\put(53.8,14.5){\circle{0.5}}

\put(15.2,14.3){\line(-1,-2){2}}
 \put(15.6,14.3){\line(1,-2){2}}
 \put(13,10){\circle{0.5}}
\put(17.8,10){\circle{0.5}}

\put(7.6,14.25){\line(0,-1){3.8}}
 \put(7.6,10){\circle{0.5}}
 \put(12.5,8.7){$v_{31}$} \put(17.5,8.7){$v_{32}$} \put(7, 8.7){$v_{33}$}

 \put(25,5){\bf Figure~2}
          \end{picture}

Moreover, it is easy to see \cite{Biyikoglu2008} that there may be many BFS graphs for a given  graphic degree sequence.
Another importation notion is majorization.
  Let $x=(x_1, x_2,\cdots, x_p)$ and $y=(y_1,y_2,\cdots, y_p)$ be
  two nonnegative integers. We arrange the entries of $x$ and $y$
  in nonincreasing order $x_{\downarrow}=(x_{[1]},\cdots x_{[p]})$ and
  $y_{\downarrow}=(y_{[1]},\cdots, y_{[p]})$. If
  \begin{equation}\sum_{i=1}^kx_{[i]}\ge \sum_{i=1}^ky_{[i]},{\rm for }\ k=1,\cdots, p,
  \end{equation}
    $x$ is said to  {\it weakly majorize} $y$ and denoted $ x\triangleright^w y$ or $y\triangleleft^w
   x$. Further, if $y\triangleleft^w
   x$ and
\begin{equation}
\sum_{i=1}^px_{[i]}= \sum_{i=1}^py_{[i]},
\end{equation} $x$ is  said to
{\it majorize} $y$ and denoted by $ x \triangleright y$ or $y\triangleleft x$.
For details, the readers are  referred to the book of Marshall   and Olkin \cite{marshall1979}.
 It is well known that the following
 result holds (see \cite{erdos1960} or \cite{wei1982}).

 \begin{proposition}(\cite{erdos1960},\cite{wei1982})
 \label{prop}
 Let $\pi=(d_0, \cdots d_{n-1})$ and $\pi^{\prime}=(d_0^{\prime}, \cdots,
 d_{n-1}^{\prime})$ be two different nonincreasing graphic  degree sequences. If
 $\pi\triangleleft \pi^{\prime},$ then there exists a series
 graphic degree sequences  $\pi_1, \cdots, \pi_k$ such that
 $\pi \triangleleft \pi_1\triangleleft \cdots \triangleleft
\pi_k\triangleleft \pi^{\prime},$ and only two components of $\pi_i$
and $\pi_{i+1}$
  are different from 1.
  \end{proposition}

 Let $G = (V,~E)$ be a simple graph with vertex set
$V(G)=\{v_1,\cdots, v_n\}$ and edge set $E(G)$. There are many matrices associated with
a graph.  Let $A(G)=(a_{ij})$ be the
$(0,1)$ {\it  adjacency matrix} of $G$, where $a_{ij}=1$ if $v_i$ and $v_j$ are adjacent, and 0 elsewhere.  The spectral radius of $A(G)$ is denoted by $\rho(G)$ and called the
{\it spectral radius} of $G$. Moreover, let $f$ be the unit eigenvector of $A(G)$ corresponding to $\rho(G)$. By the Perron-Frobenius theorem, $f$ is unique and positive. So $f$ is called  {\it Perron  vector} of $G$.
 In addition, denote by $d(v_i)$
the {\it degree} of vertex $v_i$ and  $D(G)=diag(d(u), u\in V)$
the diagonal matrix of vertex degrees of $G$, thus the matrix
$L(G)=D(G)-A(G)$ is called the {\it Laplacian matrix} of a graph
$G$. The spectral radius of $L(G)$ is equal to the largest eigenvalue of
$L(G)$  and denoted by $\lambda(G)$.  Moreover, $\lambda(G)$ is
called the {\it Laplacian spectral radius} of $G$. $\mathcal{L}(G)=D^{-1/2}L(G)D^{-1/2}$
is call the {\it normal Laplacian matrix} of a graph. The spectral radius of $\mathcal{L}(G)$ is denoted by $\mu(G)$ and called the {\it normal Laplacian spectral radius}.
In addition, the matrix $Q(G)=D(G)+A(G)$ is called the {\it signless Laplacian matrix} of $G$. The spectral radius of $Q(G)$ is denoted by $q(G)$ and called the {\it signless Laplacian spectra radius}. Let $g$ be the unit eigenvector of $Q(G)$ corresponding to eigenvalue $q(G)$ and be called {\it signless Laplacian Perron vector} of $G$.
The adjacency, Laplacian, normal Laplacian  and signless Laplacian matrices play a key role in spectral graph theory. The book `` Spectra of graph" \cite{Cvetkovicdoobsachs1995} by Cvetkovi\'{c}, Doob abd Sachs focuses on the adjacency matrix. The book `` Applications of combinatorial matrix theory to Laplacian matrices of graphs" \cite{Molitierno2012} by Molitierno  studies many properties of Laplacian matrices.  The book `` Spectral graph theory" \cite{chung1997} by  Chung focuses on normal Laplacian matrix. Moreover, there are several surveys on the topic of signless Lapacian matrices (see \cite{cvetkovic2007} and the references therein).

\section{Eigenvalues with given  tree degree sequences}

If a nonnegative sequence  $\pi=(d_0, \cdots, d_{n-1})$ is graphic  and $\sum_{i=0}^{n-1}d_i=2(n-1)$,  $\pi$ is called {\it tree degree sequence} and denoted by $\mathcal{T}_{\pi}$ the set of all trees having degree sequence $\pi$.

For a tree degree sequence $\pi=(d_0, d_1,\cdots,
d_{n-1}) $ with $d_0\ge d_1\ge\cdots\ge d_{n-1}$,  With the aid of breadth-first search
method, we can define a special tree $T_{\pi}^*$ with degree sequence $\pi$ as
follows. Assume that $d_m>1$ and $d_{m+1}=\cdots= d_{n-1}=1$ for $
0\le  m< n-1$. Put $s_0=0$. Select a vertex $v_{01}$ as a root and
begin with $v_{01}$ in layer 0. Put $s_1=d_0$ and select $s_1$
vertices $\{v_{11},\cdots, v_{1,s_1}\}$ in layer 1 such that they
are adjacent to $v_{01}$. Thus $d(v_{01})=s_1=d_0$. We continue to
construct all other layers by recursion. In general, put
$s_t=d_{s_0+s_1+\cdots+s_{t-2}+1}+\cdots+d_{s_0+s_1+\cdots
+s_{t-2}+s_{t-1}}-s_{t-1}$  for $t\ge 2$ and assume that all
vertices in layer $t$ have been constructed and are denoted by
$\{v_{t1},\cdots, v_{ts_t}\}$  with
$d(v_{t-1,1})=d_{s_0+\cdots+s_{t-2}+1},\cdots,
d(v_{t-1,s_{t-1}})=d_{s_0+\cdots+s_{t-1}}.$ Now using the induction
hypothesis, we construct all vertices in layer $t+1$. Put
$s_{t+1}=d_{s_0+\cdots+s_{t-1}+1}+\cdots+d_{s_0+\cdots +s_t}-s_t.$
Select $s_{t+1}$ vertices $\{ v_{t+1, 1},\cdots v_{t+1, s_{t+1}}\}$
in layer $t+1$ such that $v_{t+1,i}$ is adjacent to $v_{tr}$ for
 $r=1$ and $1\le i\le d_{s_0+\cdots+s_{t-1}+1}-1$ and for
$2\le r\le s_t$ and $d_{s_0+\cdots + s_{t-1}+1}+ d_{s_0+\cdots
+s_{t-1}+2}+\cdots+d_{ s_0+\cdots s_{t-1}+r-1}-r+2\le i\le
 d_{s_0+\cdots +s_{t-1}+1}+
d_{s_0+\cdots + s_{t-1}+2}+\cdots+d_{ s_0+\cdots +s_{t-1}+r}-r$.
Thus $d(v_{tr})=d_{s_0+\cdots+s_{t-1}+r}$ for $1\le r\le s_t$.
 Assume
that $m=s_0+\cdots +s_{p-1}+q$. Put $s_{p+1}=
d_{s_0+\cdots+s_{p-1}+1}+\cdots+d_{s_0+\cdots +s_{p-1}+q}-q$
 and select $s_{p+1}$
vertices $\{v_{p+1,1},\cdots, v_{p+1,s_{p+1}}\}$  in layer $p+1$
such that $v_{p+1,i}$ is adjacent to $v_{pr}$ for $1\le r\le q$
and $d_{s_0+\cdots + s_{t-1}+1}+\cdots+ d_{s_0+\cdots
+s_{p-1}+2}+\cdots+d_{ s_0+\cdots s_{p-1}+r-1}-r+2\le i\le
 d_{s_0+\cdots +s_{p-1}+1}+\cdots+
d_{s_0+\cdots + s_{p-1}+2}+\cdots+d_{ s_0+\cdots +s_{p-1}+r}-r$.
Thus $d(v_{p, i})=d_{s_0+\cdots+ s_{p-1}+i}$ for $1\le i\le q$. In
this way, we obtain a tree $T_{\pi}^*$. It is easy to see that $T_{\pi}^*$ has degree sequence $\pi$.
Moreover, this special tree $T_{\pi}^*$ is called {\it greedy tree} with a given  tree degree sequence $\pi$.
In  addition, it is easy to show the following assertion holds.

\begin{proposition}(\cite{Biyikoglu2008}, \cite{zhang2008})\label{BFRuni}
Given a tree degree sequence $\pi$,  there exists a
unique tree $T_{\pi}^*$ with degree  sequence $\pi$ having a BFS-ordering. In other words,
the greedy tree with $\pi$ is only tree having a BFS-ordering in the set $\mathcal{T}_{\pi}$.
Moreover, any two trees with the same degree sequences and having
BFS-ordering are isomorphic.
\end{proposition}

\subsection{The  (Laplacian, signless Laplacian) spectral radius}
Biyiko\u{g}lu and Leydold \cite{Biyikoglu2008}characterized all extremal trees having the largest spectral radius in the set $\mathcal{T}_{\pi}$ consisting  of all trees with a given  degree sequence $\pi$. While, Zhang \cite{zhang2008}  characterized all extremal trees having the largest (signless) Laplacian spectral radius in the set $\mathcal{T}_{\pi}$ consisting  of all trees with a given degree sequence $\pi$, since the signless Laplacian radius of a tree $T$  is always equal to its the Laplacian spectral radius.  It is not a surprise that the two extremal graphs are coincident.
\begin{theorem}(\cite{Biyikoglu2008}, \cite{zhang2008})\label{adjeigupp}
For a given  tree degree sequence $\pi$, the greedy tree $T_{\pi}^*$ is the only tree having
the largest spectral radius (resp. signless Laplacian radius) in $\mathcal{T}_{\pi}$. Moreover,
 The BFS-ordering of the greedy tree  is consistent with the Perron vector $f$ of $ G$ in such a way that $f(u) > f(v)$  (resp. $g(u)>g(v)$) implies $u\prec v.$
\end{theorem}

\begin{theorem}(\cite{Biyikoglu2008}, \cite{zhang2008})\label{majoreig}
Let $\pi$ and $\tau$ be two different tree degree sequences
with the same order. Let $T_{\pi}^*$ and $T_{\tau}^*$ have the largest spectral radii (resp. Laplacian, signless Laplacian spectral radii) in ${\mathcal T}_{ \pi}$ and ${\mathcal
T}_{\tau}$, respectively.
 If
$\pi\triangleleft \tau$, then
$\rho(T_{\pi}^*)<\rho(T_{\tau}^*)$, $\lambda(T_{\pi}^*)<\lambda(T_{\tau}^*)$  and  $q(T_{\pi}^*)<q(T_{\tau}^*)$.
\end{theorem}
 With the aid of Theorems~\ref{adjeigupp} and \ref{majoreig},
   we may deduce extremal
graphs with the largest  (resp. signless Laplacian) spectral  radius in some class of
graphs.
\begin{corollary}(\cite{Biyikoglu2008},\cite{zhang2008})\label{leaves}
Let  ${\mathcal T}_{n, s}^{(1)}$ be the set of
all trees of order $n$ with $s$ leaves and let $T_{\pi}^*$ be the greedy tree  with $\pi=(t, 2, \cdots, 2, 1, \cdots, 1)$ (i.e., if $n-1=sq+t, 0\le t<s$, then $T_{\pi}^*$ is obtained from $t$ paths of order
$q+2$ and $s-t$ paths of order $q+1$ by identifying one end of the
$s$ paths.). Then   $\rho(T)\le \rho(T_{\pi}^*)$  (resp. $\lambda(T)\le \lambda(T_{\pi}^*)$, $q(T)\le q(T_{\pi}^*)$) for any tree $T\in {\mathcal
T}_{n, s}^{(1)}$, with equality if and only if $T=T_{\pi}^*$.
\end{corollary}

\begin{corollary}(\cite{zhang2008})\label{maxdegree}
Let ${\mathcal
T}_{n,\Delta}^{(2)}$ be the set of all trees of order $n$ with the
 maximum degree $\Delta\ge 3$ and let $T_{\pi}^*$  be the greedy tree with $\pi$ which is defined as
follows:
 Denote
$p=\lceil\log_{(\Delta-1)}\frac{n(\Delta-2)+2}{\Delta}\rceil-1$ and
$n-\frac{\Delta(\Delta-1)^p-2}{\Delta-2}=(\Delta-1)r+q$ for $0\le
q<\Delta-1$. If $q=0$, put $\pi=(\Delta,\cdots,\Delta,1,
\cdots,1)$ with the number
$\frac{\Delta(\Delta-1)^{p-1}-2}{\Delta-2}+r$ of degree $\Delta$. If
$1\le q$, put $\pi=(\Delta,\cdots,\Delta,q, 1,\cdots,1)$ with the
number $\frac{\Delta(\Delta-1)^{p-1}-2}{\Delta-2}+r$ of degree
$\Delta$. Then $\rho(T)\le \rho(T_{\pi}^*)$  (resp. $\lambda(T)\le \lambda(T_{\pi}^*)$ , $q(T)\le q(T_{\pi}^*)$) for any tree $T\in {\mathcal
T}_{n, \Delta}^{(2)}$, with equality if and only if $T=T_{\pi}^*$.
\end{corollary}

\begin{corollary}(\cite{zhang2004},\cite{zhang2008})\label{inde}
Let ${\mathcal T}_{n,\alpha}^{(3)}$ be the set
of all trees of order $n$ with the independence number $\alpha$  and let $T_{\pi}^*$  be the greedy tree with $\pi=(\alpha, 2,\cdots,2,1,\cdots,1)$ the numbers $n-\alpha-1$ of 2 and  $\alpha$ of 1.
 Then $\rho(T)\le \rho(T_{\pi}^*)$  (resp. $\lambda(T)\le \lambda(T_{\pi}^*)$ , $q(T)\le q(T_{\pi}^*)$) for any tree $T\in {\mathcal
T}_{n, \alpha}^{(3)}$, with equality if and only if $T=T_{\pi}^*$.
  \end{corollary}

\begin{corollary}(\cite{guo2003}, \cite{zhang2008})\label{match}
 Let ${\mathcal T}_{n,\beta}^{(4)}$ be the set of all trees of order $n$
with the matching number $\beta$ and $T_{\pi}^*$  be the greedy tree with $\pi=(n-\beta, 2,\cdots,2,1,\cdots,1)$ and the  number $n-\beta$
of 1. Then $\rho(T)\le \rho(T_{\pi}^*)$  (resp. $\lambda(T)\le \lambda(T_{\pi}^*)$ , $q(T)\le q(T_{\pi}^*)$) for any tree $T\in {\mathcal
T}_{n, \beta}^{(4)}$, with equality if and only if $T=T_{\pi}^*$.
  \end{corollary}

On the minimum (Laplacian, signless Laplacian) spectral radius of  with a given tree degree sequence, it seems to be more difficult. The related results are referred to \cite{Biyikoglu2010}.

\subsection{The spectral radius of p-Laplacian of weighted trees}
 Now we turn to consider p-Laplacian eigenvalues of weighted graphs.
Let $G^W=(V(G), E(G), W(G))$ be a {\it weighted graph}
with vertex set $V(G)=\{v_0, v_1, \cdots, v_{n-1}\}$, edge set
$E(G)$ and weight set $W(G)=\{w_k>0,$ $k=1,$ $2,$ $\cdots,$ $|E(G)|\}$.
If $uv\in E(G), $ denote by $w_G(uv)$  the weight of an edge $uv$; if $uv \notin E(G)$,
define $w_G(uv)=0.$
The weight of a vertex $u$, denoted by $w_G(u)$, is the sum of weights
of all edges incident to $u$ in $G$.

For $p>1$,  the {\it discrete p-Laplacian} $\triangle_p(G)$ of a
function $f$ on $V(G)$ is defined to be
\begin{equation}
\triangle_p(G)f(u)=\sum\limits_{v, uv \in E(G)}(f(u)-f(v))^{[p-1]}w_G(uv),
\end{equation}
where $x^{[q]}=sign(x)|x|^q.$ When $p=2,$ $\triangle_2(G)$ is the
well-known   graph Laplacian (see \cite{Biyikoglu2009}), i.e.,
$\Delta_2(G)=L(G)=D(G)-A(G),$ where
 $A(G)=(w_G(v_iv_j))_{n \times n}$ denotes the  weighted adjacency
 matrix of $G$ and $D(G)=diag(w_G(v_0),$ $w_G(v_1),$ $\cdots,$ $w_G(v_{n-1}))$
 denotes the weighted diagonal matrix of $G$.

A real number $\lambda$ is called an {\it eigenvalue} of
$\triangle_p(G)$ if there exists a function $f \neq 0$ on $V(G)$
such that for $u \in V(G)$,
\begin{equation}
\Delta_p(G)f(u)=\lambda f(u)^{[p-1]}.\end{equation}
The function $f$ is called the {\it eigenfunction} corresponding to
$\lambda$.
 The largest eigenvalue of $\Delta_p(G)$, denoted by  $\lambda^{p}(G)$, is called the
   {\it $p$-Laplacian spectral radius}.

 Given a tree degree sequence $\pi$ and a positive set $W$, denote by $\mathcal{T}_{\pi, W}$  the set
of trees with a given  tree degree sequence $\pi$ and a positive
weight set $W$. Thus the un-weighted greedy tree $T_{\pi}^* $ has a BFS-ordering
$v_0\prec v_2\cdots\prec v_{n-1}$. Hence the unique weighted greedy tree $T_{{\pi}, W}^*$
is obtained from  the un-weighted greedy tree $T_{\pi}^* $  whose edge weighted has the following property`` for any two edges $v_iv_j$ and $v_kv_l$  with $i<j, k<l$, weight
$w(v_iv_j)\ge w(v_kv_l)$ if $i<k$ or $i=k$ and $j<l$."
Zhang and Zhang \cite{zhanggjzhang2011} proved the following
results
\begin{theorem}\cite{zhanggjzhang2011}\label{p-L-w-1} For a given  tree degree sequence $\pi$  and a positive weight set $W$, the weighted greedy tree
$T_{\pi, W}^{\ast}$  is the unique weighted tree with the largest $p$-Laplacian
spectral radius  in $\mathcal{T}_{\pi, W}$, which is independent of
$p$.
\end{theorem}

\begin{theorem}\cite{zhanggjzhang2011}\label{P-L-w-2}
Let $\pi$ and $\tau$ be two different tree degree sequences.  Let ${T}_{\pi, W}^*$ and ${T}_{\tau, W}^*
$ be two weighted greedy trees in $\mathcal{T}_{\pi, W}$ and $\mathcal{T}_{\tau, W}$ respectively.  If $\pi
\triangleleft \tau $,  then $\lambda^p(T_{\pi, W}^{\ast})<
\lambda^p(T_{\tau, W}^{\ast})$.
\end{theorem}

{\bf Remark} If the positive set $W$ consisting  of all 1,  Theorems~\ref{p-L-w-1} and \ref{P-L-w-2}  become Biyiko\u{g}lu and Leydold's result \cite{Biyikoglu2009}.  If $p=2$, Theorems~\ref{p-L-w-1} and \ref{P-L-w-2}  become Tan's result \cite{Tan2010}.  If  the positive set $W$ consisting  of all 1 and $p=2$,  Theorems~\ref{p-L-w-1} and \ref{P-L-w-2} become Zhang' result\cite{zhang2008}. In addition, Corollaries \ref{leaves}, \ref{maxdegree}, \ref{inde} and \ref{match} can be easily generalized to the weighted trees, respectively. Moreover, the p-adjacency eigenvalue may be similarly defined and Theorems~\ref{p-L-w-1} and \ref{P-L-w-2} can be translated to the p-adjacency eigenvalue.
\subsection{The Dirichlet eigenvalue of trees with boundary}

In this subsection, we consider the Dirichlet eigenvalues of  trees with boundary.
  A {\it graph with boundary} $G=(V_0\cup \partial V$,
 $E_0\cup \partial E)$ consists of interior vertex set
 $V_0$,
boundary vertex set $\partial V$, interior edge set $E_0$ that connect
interior vertices, and boundary edge set $\partial E$ that join
interior vertices with boundary vertices  (for example, see
\cite{chung1997} or \cite{friedman1993}). Throughout this subsection, we
always assume that the degree of any boundary vertex is 1 and the
degree of any interior vertex is at least 2.

A real number $\lambda^D$ is called a {\it Dirichlet eigenvalue} of $G$
  if there exists a function $f \neq 0$
such that
they satisfy the Dirichlet eigenvalue problem:
 \begin{displaymath}
\left \{ \begin{array}{ll}
L(G)f(u)=\lambda f(u) & u \in V_0;\\
f(u)=0 & u \in \partial V.\\
\end{array} \right.
\end{displaymath}
The function $f$ is called an {\it eigenfunction} corresponding to
$\lambda^D$.
 Recently, B${\i}$y${\i}$ko\u{g}lu and
Leydold \cite{biykoglu2007}  proposed the following problem:
\begin{problem}(\cite{biykoglu2007})
 Give a characterization of all graphs in a given class $\mathcal{C}$
with the Faber-Krahn property, i.e., characterize those graphs in
$\mathcal{C}$ which have minimal first Dirichlet eigenvalue for a
given ``volume''.
\end{problem}
In order to study the above problem, B${\i}$y${\i}$ko\u{g}lu and Leydold  \cite{biykoglu2007}
 extended the concept of an SLO*-ordering for describing the trees with
the Faber-Krahn property, which was introduced by Pruss (see
\cite{pruss1998}). The notation of an SLO*-ordering may be extended
for
 any connected graphs.

\begin{definition}\label{definition3.1} (\cite{biykoglu2007}) Let $G=(V_0\cup \partial V, E_0\cup
\partial E)$ be a connected graph with root $v_0$. Then a
well-ordering $\prec$ of the vertices is called spiral-like
(SLO*-ordering for short) if
 the following holds for all vertices $u, v, x, y \in V_0\cup \partial V$:

 (1) $ v\prec u $ implies $h(v)\leq h(u)$, where $h(v)$ denotes the distance between $v$ and $v_0$;

 (2) Let $uv \in E(G)$, $xy \in E(G), uy \notin E(G)$,
 $xv \notin E(G)$
with $h(u)=h(v)-1$ and $h(x)=h(y)-1$. If $u \prec x$, then $v \prec
y$ ;

 (3) If $ v \prec u$ and $ v\in \partial V$, then $u \in \partial V$.

 (4) If $ v \prec u$ and $v, u\in V_0$, then $d(v)\le d(u)$.
\end{definition}
Clearly,  if $G$ is a tree, an SLO*-ordering of $G$  is consistent
with the definition of an SLO*-ordering in \cite{biykoglu2007}.
Moreover, if there exists a positive integer $r$ such that  the
number of
 vertices $v$ with $h(v)=i+1$ is not less than the number of  vertices
$v$ with  $h(v)=i$ for $i=1, \cdots, r-1$,  and $h(u) \in \{r,
r+1\}$ for any boundary vertex $u \in
\partial V$,  $G$ is called a {\it ball approximation}.
Given a tree degree sequence $\pi$, denote by $\mathcal{T}_{\pi, B}$ the set consisting of all trees with boundary and degree sequence $\pi$.  B${\i}$y${\i}$ko\u{g}lu and Leydold  \cite{biykoglu2007} proved that  $\mathcal{T}_{\pi, B}$   contains an SLO*-tree that is uniquely determined up to isomorphism. This special tree in $\mathcal{T}_{\pi, B}$   is denoted by $T_{\pi, B}^*$.

\begin{theorem}\cite{biykoglu2007}\label{kf-tree}
For a given  tree degree sequence $\pi$,
$$\lambda^D(T)\ge \lambda^D(T_{\pi, B}^*)\  {\rm for}\  T\in \mathcal{T}_{\pi, B}$$
with equality if and only if $T=T_{\pi, B}^*$.\end{theorem}

\begin{theorem}\cite{biykoglu2007}\label{kf-tree-com}
Let $\pi$ and $\tau $  be two different tree degree sequences. If $\pi\triangleleft \tau$, then  $ \lambda^D(T_{\pi, B}^*)> \lambda^D(T_{\tau, B}^*)$.
\end{theorem}

\section{Chemical indices  for tree degree sequences}

There are many indices which are used to describe molecules and molecular compounds in chemical graph theory.
 One of the most widely known topological descriptor
 is the {\it Wiener index} which is named after chemist Wiener
 \cite{wiener1947}  who firstly considered  it.
The {\it Wiener index} of a graph $G$ is
defined as
\begin{equation}W(G)=\sum_{\{v_i, v_j\}\subseteq V(G)}d(v_i, v_j).\end{equation}
 Let $T=(V, E)$ be a rooted tree with root $r$.  For each vertex $u$, let $T(u)$ be   the subtree of
the rooted tree $T$ induced  by $u$ and all its  successors in $T$.
In other words, if $u$ is not the root $r$ of tree $T$ and $v$ is
the parent of $u$,  then $T(u)$ is the connected  component   of $T$
obtained from $T$ by deleting the edge $uv$ such that the component
does not contain the root $r$; if $u$ is the root $r$, then $T(u)$
is the tree $T$. Let $\phi_T(u)=|T(u)|$ be the number of vertices in
$T(u)$ and denote  $\phi(T)=(\phi_T(u), u\in V(T))$. We prove the following results
\begin{theorem}\cite{zhangxiang2008}\label{lar-smal}
Let $T$ be a rooted tree in ${\cal T}_{\pi}$. Then the following
conditions are equivalent:

(1) $T$ has a BFS-ordering;

(2) $\phi(T)_{\downarrow}=\phi(T_{\pi}^*)_{\downarrow}$;

(3) $T$ is isomorphic to $T_{\pi}^*$.
\end{theorem}
Further, this result can be used to characterize the trees having  the minimum
  Wiener index among all trees  with a given degree sequence.  Wang \cite{wanghua2009} independently proved the same result by a different approach.
  \begin{theorem}(\cite{wanghua2009},\cite{zhangxiang2008})\label{wienrmain}
Given a tree degree sequence $\pi$,  the greedy tree $T_{\pi}^*$  is a unique
tree with the minimum Wiener index in ${\cal T}_{ \pi}$.
\end{theorem}

\begin{theorem}\cite{zhangxiang2008}\label{wienr-com}
Let $\pi$ and $\tau$ be two different tree degree sequences. Let $T_{\pi}^*$ and $T_{\tau}^* $ be two
 the greedy trees in ${\mathcal T}_{\pi}$ and ${\mathcal
T}_{\tau}$,  respectively.
 If
$\pi\triangleleft \tau$, then $W(T_{\pi}^*)<
W(T_{\tau}^*).$
\end{theorem}
 On the other hand, it is natural to  investigate the extremal trees which attain the maximum
Wiener index among all trees with given degree sequences. But this
problem seems to be difficult, since it is discovered that the extremal
trees are not unique and depend on the values of components of degree sequences.
However, there are some partial results. A {\it caterpillar} is a tree with the property that a path remains
if all leaves are deleted and this path is called {\it spine}.
\begin{theorem}\cite{shi1993}\label{shi}
If a tree $T_{\pi}^{\sharp}$ has the maximum Wiener index among all trees in $\mathcal{T_{\pi}}$   for a given tree degree sequence $\pi$, then
$T_{\pi}^{\sharp}$  is a caterpillar.
\end{theorem}
\begin{theorem}\cite{zhangliu2010}\label{optimal-equal}
Let $\pi=(d_0, \cdots, d_{n-1})$ with $d_0\ge\cdots \ge d_{k-1}\ge 2\ge
d_{k+1}=\cdots=d_{n-1}=1$. If a caterpillar $T_{\pi}^\sharp $ has the maximum Wiener index among all trees in $\mathcal{T_{\pi}}$ and the spine is $v_0v_1\cdots v_{k-1}$ with $d(v_0)\ge d(v_{k-1})$, then  there exists a $1\le t\le k-2$ such that
$d(v_0)\ge d(v_1)\ge \cdots\ge d(v_{t})$ and $ d(v_t)\le d(v_{t+1})\le \cdots\le d(v_{k-1})$.
\end{theorem}
  Further, Schmuck, Wagner and Wang \cite{schmuckwagner2012}  proposed  and studied a new graph invariant which are based on distance.
\begin{theorem}\label{thm:nina}\cite{schmuckwagner2012}
Let $\pi=(d_0, d_1, \cdots, d_{n-1})$  be a tree degree sequence let $r$ be an arbitrary positive integer. If $p_r(T)$ is the number of pairs $(u,v)$ of vertices such that $d(u,v) \leq r$ for a tree $T\in \mathcal{T}_{\pi}$, then
$p_r(T)\le p_r(T_{\pi}^*)$ with equality  for all $r=1, \cdots, n-1$ if and only if
$T$ is the greedy tree $T_{\pi}^*$.
\end{theorem}
Recently, Schmuck, Wagner and Wang \cite{schmuckwagner2012} proposed a general graph invariant
\begin{equation}W_{\psi}(T)=\sum_{\{u, v\}\subseteq V(G)}\psi(d(u,v))\end{equation}
for any nonnegative function $\psi$.
Further, they proved the following result.
\begin{theorem}\label{generalfun}\cite{schmuckwagner2012}\cite{wagnerwangzhang2013}
Let $\pi$ be a tree degree sequence and $\psi(x)$ be any nonnegative and nondecreasing (resp. nonincreasing)  function on $x$. Then for any tree $T\in \mathcal{T}_{\pi}$,
\begin{equation}W_{\psi}(T)\ge ({\rm resp.}\ \le)\  W_{\psi}(T_{\pi}^*).\end{equation}
 Further, if  $\psi(x)$ is strictly increasing function, then equality holds if and only if $T$ is the greedy tree $T_{\pi}^*$.
\end{theorem}
{\bf Remark.~~} If $\psi(x)=x$, then $W_{\psi}(T)$ is the classical Wiener index. If
$\psi(x)=\frac{x(x+1)}{2}$, then $W_{\psi}(T)$  is the hyper-Wiener index \cite{kleinluk1995}. If $\psi(x)=\frac{1}{x}$, then $W_{\psi}(T)$ is the Harary index.
Moreover, it is easy to deduce the following result from \cite{schmuckwagner2012} and
\cite{wagnerwangzhang2013}.
\begin{theorem}
Let $\pi$ and $\tau$ be two different tree degree sequences with $\pi\triangleleft \tau$.
 If  $\psi(x)$ is any nonnegative and strictly increasing (resp. decreasing)  function on $x$, then
 \begin{equation}W_{\psi}(T_{\pi}^*)> ({\rm resp.}\ <)\  W_{\psi}(T_{\tau}^*).\end{equation}
\end{theorem}
\begin{proof} Let $p_r(T)$ be the number of pairs $\{u, v\}$ of vertices with $d(u, v)\le r$. By Theorem in \cite{wagnerwangzhang2013}, $p_r(T_{\pi}^*)\le p_r(T_{\tau}^*)$ for $r=1, \cdots, n$ with
strict  inequality for at least one $ r$.
Observe that
$$W_{\psi}(T_{\pi}^*)=\sum_{k\ge 0}(\psi(k+1)-\psi(k))(\frac{n(n-1)}{2}-p_k(T_{\pi}^*))$$
and $$W_{\psi}(T_{\tau}^*)=\sum_{k\ge 0}(\psi(k+1)-\psi(k))(\frac{n(n-1)}{2}-p_k(T_{\tau}^*)).$$
If  $\psi(x)$ is  nonnegative and strictly increasing, then $W_{\psi}(T_{\pi}^*)>   W_{\psi}(T_{\tau}^*).$
\end{proof}

\begin{theorem}\label{generalmax}\cite{schmuckwagner2012}
Let $\pi$ be a tree degree sequence and   $\psi(x)$ be
 strictly increasing and convex function  on $x$. If $T_{\pi}^\sharp$ is a tree that
maximizes $W_{\psi}(T)$ among all trees in the set $\mathcal{T}_{\pi}$, then $T_{\pi}^\sharp$ is a caterpillar.
\end{theorem}

In order to study total $\pi$-electron energy
on molecular structure, Gutman and Trinajsti\'{c} \cite{gutmantrin1972} proposed the {\it second Zagreb index}, which is defined to be
\begin{equation} ZA(T)=\sum_{uv\in E(T)}d(u)d(v),\end{equation}
where $d(u)$ and $d(v)$ are  the degrees of vertices $u$ and $v$, respectively.
\begin{theorem}\label{zagreb}\cite{schmuckwagner2012}
Let $\pi$ be a tree degree sequence. Then the greedy tree $T_{\pi}^*$ is the only tree having the maximum second zagreb index among all trees in $T\in \mathcal{T}_{\pi}$, i.e., for any tree $T\in \mathcal{T}_{\pi}$,
\begin{equation}ZA(T)\le ZA(T_{\pi}^*)\end{equation}
with equality if and only if $T$ is the greedy tree $T_{\pi}^*$.
\end{theorem}
Let $G$ be a simple graph. Denoted by $m(G, k)$ the number of $k-$matchings  in $G$ and $m(G, 0)=1$. The {\it Hosoya index}, named after Haruo Hosoya, is defined to be
\begin{equation}Z(G)=\sum_{k\ge 0}m(G, k)\end{equation}
Denote by $i(G, k)$ the number of $k-$independent sets in $G$  and $i(G, 0)=1$. The {\it Merrifield-Simmons index} is defined to be
\begin{equation}MS(G)=\sum_{k\ge 0}i(G, k)\end{equation}
The {\it energy of graph} is defined by the sum of the absolute values of all the eigenvalues of $A(G)$, and denoted by $En(G)$.
Andriantiana \cite{andriantiana2013} proved that the greedy tree has minimum energy, Hosoya index and maximum Merrifield-Simmons index.

\begin{theorem}\label{en-Z-MS}\cite{andriantiana2013}
Let $\pi$ be a tree degree sequence. Then the greedy tree $T_{\pi}^*$ is the only tree having  minimum energy, minimum Hosoya index and maximum Merrifield-Simmons indices, respectively  among all trees in $T\in \mathcal{T}_{\pi}$, i.e., for any tree $T\in \mathcal{T}_{\pi}$,
\begin{equation} Z(T)\ge Z(T_{\pi}^*),\ \ En(T)\ge En(T_{\pi}^*),\ \
 MS(T)\le MS(T_{\pi}^*),\end{equation}
with  any one equality
if and only if $T$ is the greedy tree $T_{\pi}^*$.
\end{theorem}

\begin{theorem}\cite{andriantiana2013}\label{En-Z-MS-com}
Let $\pi$ and $\tau$ be two different tree degree sequences. Then
 the greedy trees $T_{\pi}^*$ and $T_{\tau}^* $ are the  minimum energy, minimum Hosoya index and maximum Merrifield-Simmons indices
  in ${\mathcal T}_{\pi}$ and ${\mathcal
T}_{\tau}$  respectively.
 If
$\pi\triangleleft \tau$, then 
\begin{equation} En(T_{\pi}^*)>
En(T_{\tau}^*),\  \ Z(T_{\pi}^*)>
Z(T_{\tau}^*),\ \ MS(T_{\pi}^*)<
MS(T_{\tau}^*).\end{equation}
\end{theorem}
In addition, there are the following results on the number of subtrees of a tree.
 \begin{theorem}\cite{zhangxiu2013}\label{theorem2-1}
   With a given degree sequence $\pi$, $T_{\pi}^*$ is the unique
   tree with the maximum number of subtrees in ${\mathcal{T}}_{\pi}$.
  \end{theorem}

  \begin{theorem}\cite{zhangxiu2013}\label{theorem2-2}
  Given two different degree sequences $\pi$ and $\tau$. If  $\pi \triangleleft
  \tau$, then the number of subtrees of $T_{\pi}^*$ is less than
  the number of subtrees of $T_{\tau}^*$.
  \end{theorem}
  \section{Unicyclic degree sequences}
   For a given nonincreasing degree sequence $\pi=(d_0,
d_1,\cdots, d_{n-1}) $  of a unicyclic graph with $n\ge 3$, we
construct a special unicyclic graph $U_{\pi}^*$ as follows: If
$d_0=2$, denote by  $U_{\pi}^*$ the cycle of order $n$. If $d_0\ge
3$, denote by $U_{\pi}^*$ is the graph obtained from a
cycle $C_3$ of order $3$ and by breadth-first search ordering, i.e.,
 $U_{\pi}^*$ is BFS graph with the first three vertices, each pair of which is adjacent.
  Clearly, there exists only one such unicyclic graph $U_{\pi}^*$ for a given unicyclic degree sequence $\pi$, which is called {\it the greedy unicyclic graph} with a given unicyclic degree sequence $\pi$.

\begin{proposition}(\cite{hakimi1962}, \cite{senior1951})\label{graphicalunicyclic}
Let  $\pi=(d_0,\cdots, d_{n-1})$ be a positive nonincreasing integer
sequence with $n\ge 3$. Then $\pi$ is unicyclic graphic
if and only if $\sum_{i=0}^{n-1}d_i=2n$ and $d_2\ge 2$.
\end{proposition}

\begin{theorem}(\cite{liumuhuo2012}, \cite{zhang2009})\label{unicyclicmain}
Let $\pi=(d_0,d_1,\cdots, d_{n-1})$ be a unicyclic degree sequence. Then the greedy unicyclic graph $U_{\pi}^*$ is
the only  unicyclic graph having the largest (signless Laplacian)
spectral radius in ${\mathcal U}_{ \pi}$, i.e.,
\begin{equation}\rho(G)\le \rho(U_{\pi}^*), \ \ \ q(G)\le q(U_{\pi}^*)\end{equation}
with one equality if and only if $G$ is the greedy unicyclic graph $U_{\pi}^*$
\end{theorem}
\begin{theorem}(\cite{liumuhuo2012}, \cite{zhang2009})\label{uni-major}
Let $\pi$ and $\tau$ be two different degree sequences of
unicyclic graphs with the same order.
 If
$\pi\triangleleft \tau$  then $\rho(U_{\pi}^*)<
\rho(U_{\tau}^*)$ and $q(U_{\pi}^*)<
q(U_{\tau}^*)$
\end{theorem}
In order to present the results on the Dirichlet eigenvlaues, we need the following notion.
 Let $\pi=(d_0, \cdots, d_{k-1}, 1, \cdots, 1)$ be a unicyclic degree sequence  with $d_0\ge d_1\ge \dots\ge d_{k-1}$. If $d_2\ge 3$ , then let $U_{\pi, B}^*$ have SLO*-ordering  with each pair of the first three vertices  being adjacent; if $d_0=\ldots=d_{m-1}=2$ and $d_{m}=3$ for $3\le m\le k-1$, then let $U_{\pi, B}^*$ be obtained from by identifying one end vertex of a path $P_m$ with one vertex of triangle and the other end vertex of $P_m$ with the root of $T_{\tau, B}^*$ having SLO*-ordering with $\tau=(d_{m+1}-1, d_{m+2}, \cdots, d_{n-1})$; if $d_0=\ldots=d_{m-1}=2$ and $d_{m}\ge 4$ for $3\le m\le k-1$, then  let $U_{\pi, B}^*$ be obtained from by identifying one vertex of a cycle $C_m$ and the root of  $T_{\tau, B}^*$ having SLO*-ordering with $\tau=(d_{m}-2, d_{m+2}, \cdots, d_{n-1})$.    Clearly, $U_{\pi, B}^*$ is uniquely determined by $\pi$.  Then we  have the following results
\begin{theorem}\cite{zhanggj2012}\label{theorem1.1}  For a given graphic unicyclic degree sequence
$\pi=(d_0, d_1, \ldots,d_{n-1}), $ with $3\le d_0\le\ldots\le d_k $
and $d_{k+1}=\cdots=d_{n-1}=1$,
  let $G=(V_0\cup \partial V, E_0\cup \partial E)$ be a  graph with
  the Faber-Krahn property in $\mathcal{U}_{\pi}.$ Then $G$ has an
 SLO*-ordering  consistent with the first eigenfunction $f$ of
 $G$ in such a way that $ v \prec u$ implies $f(v) \geq f(u)$.
\end{theorem}

\begin{theorem}\cite{zhanggj2012}
\label{theorem1.2} For a given graphic unicyclic degree sequence
$\pi=(d_0, d_1, \ldots,d_{n-1})$
 with $3\le d_0\le\ldots\le d_k $ and $d_{k+1}=\ldots=d_{n-1}=1$, then
 $\lambda^D(G)\ge \lambda^D(U_{\pi, B}^*)$
with equality if and only if $G=U_{\pi, B}^*$. In other words,
 $U_{\pi, B}^{\ast}$  is the unique uncyclic  graph with the Faber-Krahn
property in $\mathcal{U}_{\pi}$, which can be regarded as  ball
approximation.
\end{theorem}
For $d_0=2$, we proposed the following conjecture.
\begin{conjecture}\cite{zhanggj2012}\label{conjecture5.3}
Let $\pi=(d_0, d_1, \ldots, d_{k-1}, 1, \ldots, 1)$ be a graphic
unicyclic
 degree sequence
 with $2 \leq d_0 \leq d_1\leq
\ldots \leq d_{k-1}$ and $d_k=\ldots=d_{n-1}=1$. Then
$\lambda^D(G)\ge \lambda^D(U_{\pi, B}^*)$
with equality if and only if $G=U_{\pi, B}^*$.
In other words, $U_{\pi, B}^*$ is the unique graph with the Faber-Krahn property in $\mathcal{U}_{\pi}$.

\end{conjecture}

\begin{theorem}\cite{zhanggjzhang2013}\label{theorem1.1}  For a given positive integer $k \leq n-3$,  let $U_{n,k}^{\ast}$ be unicyclic graph obtained from by identifying one end vertex of a path $P_{n-k-1}$ and one vertex of a triangle and the other end vertex of $P_{n-k-1}$ and the center of star $K_{1, k}$. Then
 for any unicyclic graph $G$ of order $n$ with $k$ pendant vertices, $\lambda^D(G)\ge \lambda^D(U_{n,k}^*)$ with equality if and only if $G=U_{n,k}^{\ast}$.
 In other words,
$U^{\ast}$ is the only
 unicyclic graph with the Faber-Krahn property in the set $\mathcal{U}_{k}$ of all unicyclic graphs of order $n$ with $k$ pendant vertices, which  can be regarded a ball.

\end{theorem}

\section{Graphic degree sequences}
In this section, we discuss some extremal properties of $\mathcal{G}_{\pi}$.

  \begin{theorem}(\cite{Biyikoglu2008},\cite{zhang2009})For a given graphic degree sequence $\pi$, if $G$  is a simple connected graph that has the largest (resp. signless Laplacian) spectral radius in $\mathcal{G}_{\pi}$, then $G$
has a BFS-ordering, i.e., $G$ is a BFS graph.
\end{theorem}
Moreover,  all extremal graphs having the (signless Laplacian) spectral radius for bicyclic (tricyclic) degree sequence have been characterized (see \cite{huangliu2011},\cite{jiang2011}, \cite{liumuhuo2012}).

On the Dirichlet eigenvalues for bicyclic degree sequences, we have the following results.
For given a graphic bicyclic degree sequence $\pi=(d_0, \cdots, d_{n-1})$ with $3\le d_0\le \cdots d_{k-1}$ and $d_k=\cdots=d_{n-1}=1$, let $G_{\pi, B}^*$ be the unique graph having SLO* ordering $v_0v_1\cdots v_{n-1}$ with the induced subgraph by $\{v_0, v_1, v_2, v_3\}$ being $K_4-v_2v_3$. Then

\begin{theorem}\cite{zhanggjzhang2012}\label{theorem 4.2}
Let $\pi=(d_0, d_1, \cdots, d_{k-1}, 1, 1, \cdots, 1)$ be a
 graphic bicyclic degree sequence with  $3\le d_0\le d_1\cdots\le
d_{k-1}$ for $k\ge 4$. Then  for any bicyclic graph $G$ with degree sequence $\pi$,  $\lambda^D(G)\ge \lambda^D(G_{\pi, B})$ with equality if and only if $G= G_{\pi, B}^{\ast}$. In other words, $G_{\pi, B}^{\ast}$  is the only  extremal graph with the smallest first Dirichlet eigenvalue
 in $\mathcal{G}_{\pi}$.
\end{theorem}

\begin{theorem}\cite{zhanggjzhang2012}
Let $\pi=(d_0, d_1, \cdots, d_{k-1}, 1, 1, \cdots, 1)$ and $\tau=(d_0', d_1', \cdots, d_{k-1}', $ $1, 1, \cdots, 1)$ be two different graphic bicyclic degree sequences. If
$\pi\triangleleft \tau$ with $d_0\ge 3,$ $d_0'\ge 3$, then
$ \lambda^D(G_{\pi, B})> \lambda^D (G_{\tau, B})$.
\end{theorem}

A graph $H$ is a {\it minor} of a graph $G$ if $H$ can be
obtained from a subgraph of $G$ by contracting edges. An $H$ minor
is a minor isomorphic to $H$. If $H$ is a complete graph, we
 say that $G$ contains a {\it clique minor} of size $|H|$.
 For a graph $G$, the
Hadwiger number $h(G)$ of G is the maximum integer $k$ such that $G$
contains a clique minor of size $k.$
Moreover, denote by  $\chi(G)$ the chromatic number of $G$. Let
\begin{equation}h(\pi)=\max\{h(G): G\in \mathcal{G}_{\pi}\}
, \ \  \chi(\pi)=\max\{\chi(G): G\in \mathcal{G}_{\pi}\}.\end{equation}
\begin{conjecture} (Hadwiger＊s Conjecture for Degree Sequences, \cite{robertson2010})
For every graphic degree sequence $\pi$, $\chi(\pi)\le h(\pi)$.
\end{conjecture}
Recently, Dvo\u{r}\'{a}k and Mohar proved a stronger result than the above conjecture.
\begin{theorem}
\cite{dvorakmohar2014} For every graphic degree sequence $\pi$,
\begin{equation}\chi(\pi)\le h^{\prime}(\pi)\le h(\pi),\end{equation}
 where
$h^{\prime}(G)$ is the maximum $k$ such that $G$ has a $K_k$-
topological minor and
 $\chi{\prime}(\pi)=\max\{h^{\prime}(G): G \in
\mathcal{G}_{\pi}\}$.
\end{theorem}


\end{document}